\documentclass[12pt]{article}

\usepackage{amssymb,bbm}    
\usepackage{amsmath}    
\usepackage{amsthm} 
\usepackage{mathtools}
\usepackage{mathrsfs}
\usepackage[x11names]{xcolor}
\usepackage{easyReview}
\allowdisplaybreaks
\usepackage{hyperref}

\newcommand{\sumi}{\sum_{i=1}^n}

\begin{document}

\theoremstyle{plain}
\newtheorem{theorem}{Theorem}
\newtheorem{proposition}{Proposition}
\newtheorem{cor}{Corollary}
\newtheorem{lemma}{Lemma}
\newtheorem{defin}{Definition}
\newtheorem{problem}{Problem}
\newtheorem{remark}{Remark}
\newtheorem{example}{Example}
\newcounter{aid}

\theoremstyle{remark}

\def\N{I\!\!N}
\def\Z{Z}
\def\Zc{\mathcal{Z}}
\def\Zs{\mathscr{Z}}
\def\D{{\cal D}}
\def\Q{Q}
\def\R{{\mathbb R}}
\def\Rd{{\mathbb R}^d}
\def\bR{I\!\!R}
\def\C{C}
\def\P{{\cal P}}
\def\F{{\cal F}}
\def\X{{\cal X}}
\def\Y{{\cal Y}}
\def\bX{{\bf X}}
\def\PROB {{\mathbb P}}
\def\EXP {{\mathbb E}}
\def\IND{{\mathbb I}}
\def\Var{{\mathbb Var}}
\def \a {^{\ast}}
\def \p {^{\prime}}
\def \pp {^{\prime\prime}}
\def \stt {\stackrel{\triangle}{=}}
\def \stL {\stackrel{L}{\longrightarrow}}
\def \stUL {\stackrel{UL}{\longrightarrow}}
\def \stD {\stackrel{{\cal D}}{\longrightarrow}}
\def \eD {\stackrel{{\cal D}}{=}}
\def \lD {\stackrel{{\cal D}}{\le}}
\def \ve {\varepsilon}
\def \nl {\nolimits}
\def\argmin{\mathop{\rm arg\, min}}
\def\argmax{\mathop{\rm arg\, max}}
\def\dist{\nu}
\def\marg{\mu}
\newcommand{\nrm}[1]{\left\Vert #1 \right\Vert}
\newcommand{\eps}{\varepsilon}
\newcommand{\tS}{\tilde S}
\newcommand{\Ber}{\text{Bernoulli}}
\newcommand{\tp}{\tilde p}
\newcommand{\dP}{R}
\renewcommand{\k}{k_{n}}
\newcommand{\gproto}{\tilde g}
\newcommand{\gkproto}{\hat{g}}
\newcommand{\proto}{\text{Proto-NN}}
\newcommand{\optinet}{\text{OptiNet}}
\newcommand{\hybrid}{\text{Proto-$k$-NN}}
\newcommand{\ft}{f^t}
\newcommand{\G}{\mathcal G}
\newcommand{\nutilde}{{\widetilde \nu}}
\newcommand{\mhat}{{\widehat m}}
\newcommand{\mphat}{{\mhat\p}}
\newcommand{\defeq}{=}
\renewcommand{\H}{\mathcal H}

\newcommand{\LG}[1]{\textcolor{red}{\textbf{LG:} #1}}
\newcommand{\MK}[1]{\textcolor{blue}{\textbf{MK:} #1}}


\title{On public and private binary classification with metric space valued predictors}
\author{
L\'aszl\'o Gy\"orfi\thanks{Department of Computer Science and Information Theory, Budapest University of Technology and Economics, Magyar Tud\'{o}sok krt. 2., Budapest, H-1117, Hungary.
\texttt{gyorfi@cs.bme.hu}}
\and Martin Kroll\thanks{Fakultät für Mathematik, Physik und Informatik, Universität Bayreuth, Universitätsstraße 30, 95447 Bayreuth, \texttt{martin.kroll@uni-bayreuth.de}.}
\and Harro Walk\thanks{Institute for Stochastics and Applications, University of Stuttgart, Pfaffenwaldring 57, D-70569 Stuttgart, Germany, \texttt{harro.walk@mathematik.uni-stuttgart.de}.}
}

\maketitle

\begin{abstract}
We consider the problem of binary classification in a framework where the predictor $X$ takes values in an arbitrary separable metric space $\X$ and the label $Y$ values in $\{ \pm 1 \}$.
In the first part of this work, we assume that one has direct access to an i.i.d.~sample $(X_1,Y_1),\ldots,(X_n,Y_n)$ from the unknown distribution of the pair $(X,Y)$.
We derive a convergence rate for the {\sc Proto-NN} classifier which was recently introduced in \cite{GyWe21} as a classifier in the presence of metric space-valued predictors.
In the second part of the paper, we reconsider the same problem under an additional privacy constraint.
More precisely, we work in the framework of local differential privacy where one assumes that the data $(X_1,Y_1),\ldots,(X_n,Y_n)$ cannot be directly observed but only a privatised surrogate obtained through a suitable mechanism satisfying the privacy constraint is available.
The statistician should select an optimal privacy mechanism from the class of all mechanism that guarantee local differential privacy. 
Our method of choice is to add Laplace distributed noise to both a set of indicator functions corresponding to a random partition of the metric space $\X$ and the labels $Y_i$.
The data obtained by this perturbation approach satisfy the constraint of local differential privacy and we show that a privatised version of the {\sc Proto-NN} classifier using the privatised data only is universally consistent.
Finally, a rate of convergence for the privatised {\sc Proto-NN} classifier is derived.
\end{abstract}

\noindent


\noindent



\section{Introduction}

Let $(\X,\rho)$ be a separable metric space equipped with its Borel $\sigma$-field. A standard reference for such probability measures on metric spaces is the monograph \cite{Par67} by Parthasarathy.
Consider a random pair $(X,Y)$ where we assume that the predictor $X$ takes values in $\X$ and the binary label
$Y$ values in $\{\pm 1\}$.
Let $\mu$ be the distribution of $X$, that is, $\mu(A)= \PROB (X\in A)$ for all Borel sets $A$.

The standard setup of non-parametric classification is to decide on $Y$ given $X$, that is, one aims to find a decision function $g$ defined on the range of $X$ such
that $g(X)=Y$ with large probability.  
Let
\[
L(g)=\PROB\{g(X)\ne Y\}.
\]
denote the \emph{error probability} of the decision $g(X)$.
It is well-known that the error probability is minimized by the \emph{Bayes decision rule} $g^*$,
\[
g^*(x) =\mbox{sign }m(x),
\]
where $\mbox{sign}(x) = \IND_{\{x\geq 0\}}- \IND_{ \{ x <0\}}$, and
\begin{equation*}
	\label{def:model}
	m(x) = \EXP [Y | X=x]
\end{equation*}
is the \emph{regression function} which is well-defined for $\mu$-almost all $x \in \X$.
In the following we denote with
\[
L^*=\PROB\{g^*(X)\ne Y\}=\min_g L(g)
\]
the smallest possible error probability which is attained by the Bayes decision rule.
The quantity $L^\ast$ provides a theoretical benchmark which is in general not achievable since the optimal Bayes decision rule depends on the in practise unknown distribution of the pair $(X,Y)$.

Instead, in classification it is commonly assumed that raw data $\D_n$
\begin{align}
\label{eq:raw_data}
\D_n \defeq \{(X_1,Y_1),\ldots,(X_n,Y_n)\}
\end{align}
consisting of i.i.d.\,copies of the random pair $(X,Y)$ are available.

The main aim of this paper is to propose classifiers $g$ based on either the raw data \eqref{eq:raw_data} or on suitably privatised surrogates.
The case when the raw data \eqref{eq:raw_data} are available has been intensely studied in the literature, especially in the case when the predictor $X$ belongs to some Euclidean $\R^d$.
For this special case, let us mention the work \cite{KoKr07} of Kohler and Krzy\.zak where plug-in decision rules based on local averaging, partitioning, and nearest neighbour estimates of the regression function are studied in detail.
The present work deals with plug-in classifiers for partitioning estimates of the regression function only since this approach can rather easily be adapted to suggest privatised data and classifiers based on these data.
In the Euclidean case $\X=\R^d$, Berrett and Butucea \cite{BeBu19} studied a classifier, where the corresponding partition is a cubic one.
However, for a general metric space $\X$ it is not obvious how to define a partition.
In order to address this problem, we revisit the {\sc Proto-NN} classifier introduced recently in Gy\"orfi and Weiss \cite{GyWe21}.
That work suggests to use a special data dependent partition which is defined as the Voronoi partition obtained from a set of so-called prototypes.
In the case of $\X=\R^d$ and for non-privatised data, there exist many works on the consistency of data dependent partitioning, Biau, Devroye and Lugosi \cite{BiDeLu12},
Breiman et al. \cite{BrFrStOl84},
Devroye et al. \cite{DeGyLu96},
Lugosi and Nobel \cite{LuNo96}, just to mention a few.
However, the results in the mentioned papers do not yield the optimal rate of convergence.
In order to fill this existing gap in the literature, we consider the case with available raw data \eqref{eq:raw_data} first and derive a convergence rate for the {\sc Proto-NN} classifier under some smoothness and margin condition on the regression function $m$ (see Theorem \ref{patt_rate} in Section~\ref{s:non-private}).
Subsequently, in Section~\ref{s:private}, we turn to the case of privatised data. Under the local differential privacy framework, we first state universal consistency of the {\sc Proto-NN} classifer (Theorem \ref{thm:multi_label_class,priv}) and then derive its rate of convergence (Theorem \ref{p_patt_rate}).
All proofs are deferred to Section~\ref{s:proofs}.


\section{Classification from non-private data}\label{s:non-private}

For non-private data, Gy\"orfi and Weiss \cite{GyWe21} introduced the so called {\sc Proto-NN} classifier, which is a partitioning classifier with data-driven partition.
For an integer $k\ge 1$, we assume that in addition to the labelled sample $\D_n$, we also have an independent unlabelled sample, called prototypes, $\bX'_k =\{X'_1,\dots ,X'_k\}$, where the $X'_i$'s are independent copies of $X$.
 Sometimes we will write $k_n$ instead of $k$ in order to emphasize that the number $k$ of prototypes is usually chosen in dependence on the sample size $n$.
Let the data-driven partition $\P_{k}$ of $\X$ be such that $\P_{k}$ is a Voronoi partition with the nucleus set $\bX'_k$, i.e., 
\begin{align}
\label{Pk}
\P_{k}=\{A_{k,1},A_{k,2},\dots ,A_{k,k}\}
\end{align}
such that $A_{k,j}$ is the Voronoi cell around the nucleus $X'_j$,
\begin{align*}
A_{k,j} = \left\{x\in\X : j = \argmin_{1\leq i \leq k} \rho(x,X'_i)\right\},
\end{align*}
where tie breaking is done by indices, i.e., if $X'_i$ and $X'_j$ with $i \neq j$ are equidistant from $x$, then $X'_i$ is declared ``closer'' if $i < j$.
In this paper we assume that ties appear with zero probability.

As in \cite{GyWe21}, an additional assumption on the tie breaking is introduced.
For any fixed $x \in \X$, let $H_x\colon [0,\infty) \to [0,1]$ be defined by
\begin{equation}\label{defhx}
H_x(r) =\PROB \left(\rho (x,X)\le r\right)=\mu(S_{x,r}), \quad r \ge 0,
\end{equation}
be the cumulative distribution function of $\rho (x,X)$ where $S_{x,r}$ denotes the closed ball of radius $r$ centered at $x$. In the sequel, we assume that $H_x$ is continuous for all $x \in \X$.
This assumption holds, for example, in the case that $\X = \R^d$ and  $X$ has a density.
If $H_x$ is continuous, then in the definition of nearest neighbors, ties happen with probability zero.
In general, one can achieve that $H_x$ is continuous by adding a randomized component to $X$, see Gy\"orfi and Weiss \cite{GyWe21}.

For any fixed value of $k$, define the function $L_k\colon \X \to \P_{k}$ via
\begin{align*}
L_k(x)
&=A_{k,j},
\end{align*}
when $x\in A_{k,j}$.
For $\X=\R^d$, assuming that $\mu$ has a density, Devroye et al. \cite{DeGyLuWa17} and Gibbs and Chen \cite{GiCh20} proved that
$k\mu(A_{k,1})$ and  $k\mu(L_k(x))$ have limit distributions that depend neither on $\mu$ nor on $x$, but only on the dimension $d$.

Set
\begin{align*}
\nu_{n}(A_{k,j}) =\frac 1n \sum_{i=1}^n Y_i  \IND_{\{X_i \in A_{k,j}\}}
\end{align*}
Then, the {\sc Proto-NN} classification rule is defined by
\begin{align*}
g_{n}(x)=\mbox{sign } \nu_{n}(A_{k,j}),\quad \mbox{ if } x\in A_{k,j}.
\end{align*}
Under the conditions 
\begin{align}
\label{kn}
k_n \to \infty \quad \mbox{and}\quad n/k_n\to \infty,
\end{align}
the strong universal consistency of the classifier $g_n$ has been proved in Corollary 4 of Gy\"orfi and Weiss \cite{GyWe21}.

In order to get a nontrivial rate of convergence for the classification error probability, we impose a smoothness condition and a margin condition.

The regression function $m$ satisfies the {\textit{generalized Lipschitz condition}} if there is a monotonically increasing function $h: [0,1]\to \R^+$ with $h(s)\downarrow 0$ as  $s\downarrow 0$ such that
\begin{equation}
\label{gen-Lip}
|m(x)-m(z)|\le h(\marg(S_{x,\rho(x,z)}))
\end{equation}
for any $x,z \in \X$, see Gy\"orfi and Weiss \cite{GyWe21}.

If $\X=\R^d$, then the generalized Lipschitz condition appears in  Chaudhuri and Dasgupta \cite{ChDa14}
and in  D\"oring, Gy\"orfi and Walk \cite{DoGyWa18} with
\begin{align}
\label{hs}
h(s)=c_0\cdot s^{1/d}.
\end{align}
If  $\P_h=\{A_{h,1},A_{h,2},\ldots\}$ is a cubic partition of $\Rd$ with cubic cells $A_{h,j}$ of volume $h^d$, $X$ is bounded,  and the margin and the Lipschitz conditions on $m$ are satisfied, 
Kohler and Krzy\.zak \cite{KoKr07} showed for partitioning binary classification with suitably chosen $h_n$ that
\begin{align}
	\label{99'}
	\EXP\{L(g_n)\}-L^*
	&=
	O\left(n^{-\frac{1+\gamma}{3+\gamma+d}}\right).
\end{align}
We say that the strong density assumption (SDA) holds, when for all $\mu(A_{h,j})>0$, we have
\begin{equation}\label{eq:SDA}
	\mu(A_{h,j}) \geq c h^d, \qquad j = 1,2,\ldots,
\end{equation}
for some constant $c > 0$.
If $X$ is bounded, the margin and the Lipschitz conditions on $m$ are satisfied and the SDA is met, then the Kohler and Krzy\.zak showed that the order of the rate of convergence is
\begin{align}
	\label{9'}
	n^{-\frac{1+\gamma}{2+d}}.
\end{align}
Under the SDA, Audibert and Tsybakov \cite{AuTs05} proved the minimax optimality of this rate.

Concerning the rate of convergence of binary classification,  Audibert and Tsybakov \cite{AuTs05}, Mammen and Tsybakov \cite{MaTs99} and Tsybakov \cite{Tsy04}
discovered and investigated the phenomenon that there is a dependence on the behaviour of $m$ in the neighbourhood of the decision boundary 
\begin{align*}
	B^*=\{x: m(x)=0\}.
\end{align*}
The {\em margin condition} states that for all $0<t\le 1$, we have
\begin{align}
	\label{wtsyb}
	G^*(t)
	&:=
	\int \IND_{\{0<|m(x)|\le t\}}\mu(dx)
	\le
	c^*\cdot t^{\gamma},
\end{align}
for some constants $c^* \geq 0$ and $\gamma \geq 0$.

\begin{theorem}
\label{patt_rate}
Assume that the distribution function $H_x$ is continuous for all $x \in \X$, 
the regression function $m$ satisfies the generalized Lipschitz condition such that the function  $h^{1+\gamma}$ is concave and the margin condition \eqref{wtsyb} holds for some $\gamma > 0$. 
Then 
{\begin{align*}
\EXP\{L(g_n)\}-L^*
&=
O\left(h(1/k)^{1+\gamma} \right)
+
O\left( \sqrt{\frac{k}{n }}\right).
\end{align*}}
\end{theorem}

\noindent
{\sc Remark 1.}
We conjecture that the bound in Theorem  \ref{patt_rate} can be improved such that
\begin{align}\label{eq:conjecture:after:thm1}
\EXP\{L(g_n)\}-L^*
&=
O\left(h(2/k)^{1+\gamma} \right)
+
O\left(\left( \frac{k}{n }\right)^{(\gamma+1)/2}\right).
\end{align}
Under the conditions of Theorem  \ref{patt_rate}, Gy\"orfi and Weiss \cite{GyWe21}
bounded the rate of convergence of $K$-NN classification rule $\widehat g_n$ .
Their Theorem~6 states that
\begin{align}\label{KNN}
\EXP\{L(\widehat g_n)\}-L^*
&=
O\left(h(2K/n)^{1+\gamma} \right)
+
O\left(\left( \frac{1}{K }\right)^{(\gamma+1)/2}\right).
\end{align}
Thus, for the correspondence
\begin{align}\label{kK}
K
&=
\lfloor n/k \rfloor,
\end{align}
the conjecture \eqref{eq:conjecture:after:thm1}  and the rate \eqref{KNN}  are equivalent.\\

\noindent
{\sc Remark 2.}
If, in the special case of $\X=\R^d$, (\ref{hs}) holds, then
for the choice 
\begin{align*}
k\approx n^{\frac{d}{d+2(1+\gamma)} }
\end{align*}
Theorem  \ref{patt_rate} implies
\begin{align}
	\label{0'}
	\EXP\{L(g_n)\}-L^*
	&=
	O\left(n^{-\frac{1+\gamma}{d+2(1+\gamma)}}\right),
\end{align}
which is between (\ref{99'}) and (\ref{9'}), for $\gamma\le 1$. 

\section{Classification from privatised data}\label{s:private}

In the framework of local differential privacy, the raw data $\mathcal D_n$ in \eqref{eq:raw_data} are not directly accessible but only a suitably anonymized surrogate.
The only restriction on the class of potential privacy mechanisms is the condition of local differential privacy which has to be satisfied by the anonymized data.
The choice of the privacy mechanism in the present work is motivated by the recent paper Berrett and Butucea \cite{BeBu19} where a first step in this direction was done for the Euclidean case.
In this special case with $\X=\R^d$, let $\P_h=\{A_{h,1},A_{h,2},\ldots\}$ be a cubic partition of $\Rd$ with cubic cells $A_{h,j}$ of volume $h^d$.
In the privacy setup in \cite{BeBu19}, the dataholder of the $i$-th datum $X_i$ generates and transmits to the statistician the data
\begin{align*}
	Z'_{i,j} \defeq  Y_i\IND_{\{X_i \in A_{h,j}\}} + \sigma_Z \epsilon_{i,j}, \quad j = 1,2,\ldots
\end{align*}
with noise level $\sigma_Z>0$, and $\epsilon_{i,j}$ ($i=1,\ldots,n$, $j=1,2,\ldots$) are independent centred Laplace random variables with unit variance.
This means that the $i$-th individual generates noisy data for any cell $A_{h,j}$.
This privatization mechanism was studied already
for classification in Berrett and Butucea \cite{BeBu19}, 
in  Berrett,  Gy\"orfi and  Walk \cite{BeGyWa21},
and in Cs\'aji et al. \cite{CsGyTaWa24}.
For regression and density estimation this privacy mechanism was considered in Gy\"orfi and Kroll \cite{GyKr25} and \cite{GyKr23}, respectively.

Let us briefly recall the definition of LDP.
A non-interactive privacy mechanism can be described by the conditional distributions $Q_i$ of  the privatized data $Z_i$ given the raw data $(X_i,Y_i)$, $i = 1,\dots ,n$, where $Z_i$ takes its values in a measurable space $(\Zc,\Zs)$. Specifically, given a realization of the raw data $(X_i, Y_i) = (x_i, y_i)$, one generates $Z_i$
according to the probability measure defined by $Q_i(A \,|\, (X_i, Y_i) = (x_i, y_i))$, for any $A \in \Zs$. Such a non-interactive mechanism is {\em local} since any data
holder can independently generate privatized data (that is, without a trusted third party). For a privacy parameter $\alpha \in [0,\infty]$, a non-interactive privacy mechanism is said to be an $\alpha$-{\em locally differentially private} mechanism if the condition
\vspace{1mm}
\begin{equation*}
    \frac{Q_i(A \,|\, (X_i, Y_i) = (x, y)) }{Q_i(A \,|\, (X_i, Y_i) = (x', y')) } \leq \exp(\alpha)
    \vspace{1mm}
\end{equation*}
holds for all $A \in \Zs$ and all realizations $(x, y)$, $(x', y')$ of the raw data. 
The noise level $\sigma_Z$ has to be chosen as $2 \sqrt{2}/ \alpha$ to make the overall mechanism satisfy $\alpha$-LDP, see \cite{BeBu19}.

In this setup, the privacy mechanism has been derived from a rather natural partition of Euclidean space into cubes of equal volume.
For general metric spaces, a partition cannot be defined this way since the notion of a cubes is not well-defined. 
Instead, we rely in the following on the Voronoi partition defined by the prototypes from Section~\ref{s:non-private}.
Assume that the prototypes $\bX'_k$ are published in advance and, therefore that the corresponding Voronoi partition (\ref{Pk}) is known to all the dataholders. 
Then, privatised data are generated via the non-interactive mechanism given by
\begin{equation}
	\label{Eq:Mech1M}
	Z_{i,j} \defeq  Y_i \IND_{\{X_i \in A_{k,j}\}} + \sigma_Z \epsilon_{i,j}, \quad j = 1,2,\ldots
\end{equation}
In terms of the so-defined privatised data, we put
\begin{align*}
\widetilde \nu_{n}(A_{k,j})=\frac 1n \sum_{i=1}^n Z_{i,j}, \quad j = 1,2,\ldots
\end{align*}
and
\begin{align*}
\widetilde g_{n}(x)=\mbox{sign } \widetilde \nu_{n}(A_{k,j}),\quad \mbox{ if } x\in A_{k,j}.
\end{align*}

The following theorem states the universal consistency of the  private classifier $\widetilde g_n$, extending the corresponding result from Berrett and Butucea \cite{BeBu19} for the Euclidean case to general metric spaces.

\begin{theorem}
\label{thm:multi_label_class,priv}
If 
\begin{align}
\label{nkn}
k_n \to \infty \quad \mbox{and}\quad k_n/\sqrt{n}\to 0,
\end{align}
holds as $n \to \infty$, then the privatized classification rule $\widetilde g_n$ is universally strongly consistent, i.e., for any distribution of $(X,Y)$
\begin{align*}
\lim_{n\to\infty} L(\widetilde g_n)=L^*
\qquad \text{a.s.}
\end{align*}
\end{theorem}

Next, we bound the rate of convergence of the excess error probability for private data.

\begin{theorem}
\label{p_patt_rate}
Assume that  the distribution function $H_x(\cdot)$ is continuous for each $x$,
$m$ satisfies the generalized Lipschitz condition such that the function  $h(\cdot )^{1+\gamma}$ is concave and the margin condition \eqref{wtsyb} holds for some $\gamma > 0$.
Then,
\begin{align*}
\EXP\{L( \widetilde g_n)\}-L^*
&=
O\left(h(1/k)^{1+\gamma} \right)
+
O\left(\frac{k}{\sqrt{n} }\right).
\end{align*}
\end{theorem}

\noindent
{\sc Remark 3.}
The second term of the bound in this theorem corresponds to the estimation error such that the factor is proportional to $\sigma_Z$:
\begin{align*}
O\left(\frac{\sigma_Z k}{\sqrt{n} }\right)
&=
O\left(\frac{k}{\sqrt{n\alpha^2} }\right),
\end{align*}
where $n\alpha^2$ is usually interpreted as the effective sample size.\\

\noindent
{\sc Remark 4.}
If in the special case of $\X=\R^d$ the condition (\ref{hs}) holds, then
the bound in Theorem  \ref{p_patt_rate} has the form
\begin{align*}
O\left(1/k^{(1+\gamma)/d} \right)
+
O\left( \frac{k}{\sqrt{n}}\right).
\end{align*}
Set
\begin{align*}
k\approx n^{\frac{d}{2(d+1+\gamma)} }.
\end{align*}
Then, Theorem  \ref{p_patt_rate} implies
\begin{align*}
\EXP\{L(g_n)\}-L^*
&=
O\left(n^{-\frac{1+\gamma}{2(d+1+\gamma)}}\right).
\end{align*}

\section{Proofs}\label{s:proofs}

\subsection{Proof of Theorem \ref{patt_rate}}

\begin{proof}
If $x\in A_{k,j}$, then let
\begin{align*}
m_{n}(x) 
\doteq \frac{\nu_{n}(A_{k,j})}{\mu(A_{k,j})},
\end{align*}
for which we have $g_n(x) = \mbox{sign } m_{n}(x)$.
Set $\nu(A_{k,j}) = \int_{A_{k,j}}m(z)\mu(dz)$ and
\begin{align*}
\overline{m}_{k}(x)
&=
\EXP\{m_{n}(x)\mid \bX'_k\}
=
\frac{\nu(A_{k,j})}{\mu(A_{k,j})}
\end{align*}
Let us introduce the two auxiliary decision rules
\begin{align*}
\overline{g}_{k}(x)
&=
\mbox{sign } \overline{m}_{k}(x)
\end{align*}
and
\begin{align*}
g_{\mathrm{NN}}(x)
&=
\mbox{sign } m(X'_{k,1}(x)),
\end{align*}
where $X'_{k,1}(x)$ stands for the 1-NN of $x$ from $X'_{1},\dots ,X'_{k}$.
We decompose the excess error probability as
\begin{align*}
\EXP\{L(g_n)\}-L^*
&=
I_n+J_n,
\end{align*}
with the approximation error
\begin{align*}
I_n
&=
\EXP\{L(g_{\mathrm{NN}})\}-L^*
\end{align*}
and the estimation error
\begin{align*}
J_n
&=
\EXP\{L(g_n)\}-\EXP\{L(g_{\mathrm{NN}})\}.
\end{align*}
It is well-known that for any decision rule $g$ it holds that
\begin{align}
L(g)-L^*
&=\int\IND_{\{ g( x)\ne g^*(x)\}}|m(x)|\mu(dx),
\label{g}
\end{align}
cf.~Theorem~2.2 in Devroye, Gy\"orfi and Lugosi \cite{DeGyLu96}.
For the {\it approximation error}, taking $g=g_{\mathrm{NN}}$, \eqref{g} implies
\begin{align*}
I_{n} 
&= \int  \EXP\left\{ 
\IND_{\{ \mbox{sign } m(X'_{k,1}(x))\ne \mbox{sign } m(x)\}} \right\}|m(x)|\mu(dx)\\
&\le \int  \EXP\left\{|m(X'_{k,1}(x)) -m(x)| 
\IND_{\{0<|m(x)| \le |m(X'_{k,1}(x)) -m(x)|\}} \right\}\mu(dx).
\end{align*}
For $t>0$, define the convex function $g_t$ on $\R_+$ by
\begin{align*}
g_t(v)=(2v-t)_+.
\end{align*}
Since
\begin{align*}
v\IND_{\{t\le v\}}
&\le
g_t(v), \quad v\ge 0,
\end{align*}
one has
\begin{align}
I_{n}
&\le 
\int_{\{0<|m(x)| \}}  \EXP\left\{g_{|m(x)| }\left(|m(X'_{k,1}(x)) -m(x)| \right) \right\}\mu(dx).
\label{20'}
\end{align}
By the generalized Lipschitz condition,
\begin{align*}
|m(x) -m(X'_{k,1}(x))|
&\le 
h(\mu(S_{x,\rho(x,X'_{k,1}(x))})).
\end{align*}
The distribution function $H_x$ is continuous for any $x \in \X$,
therefore, as in  Biau and Devroye \cite{BiDe15}
and in  Gy\"orfi and Weiss \cite{GyWe21},
\begin{align}
\label{mD}
\mu ( S_{x,\rho(x,X'_{k,1}(x))})
\eD \min_{1\le i\le k} U_i
\end{align}
with $U_1,\dots ,U_k$ i.i.d.\ uniform on $[0,1]$.
Thus, for all $x \in \X$,
\begin{align}
\EXP\left\{g_{|m(x)|}\left(\left|m(x) -m(X'_{k,1}(x))
\right| \right)\right\}
&\le
\EXP\left\{g_{|m(x)|}\left(h(\mu(S_{x,\rho(x,X'_{k,1}(x))})) \right)\right\}\nonumber\\
&=
\EXP\left\{g_{|m(x)|}\left( h\left(\min_{1\le i \le k}U_i\right)\right)\right\}.
\label{20'''}
\end{align}
Now, (\ref{20'}) and (\ref{20'''}) yield
\begin{align}
I_{n} 
&\le
 \EXP\left\{\int_{\{0< |m(x)|\}}   
g_{|m(x)| }\left( h\left(\min_{1\le i \le k}U_i\right)  \right)\mu(dx)\right\}.
\label{20*}
\end{align}
The margin condition yields
\begin{align*}
\PROB\{0<|m(X)|\le t\}
\le 
G^*(t)
\le \min\{c^*t^{\gamma},1\},
\end{align*}
when $ 0 < t $.
For $v\in\R_+$, by partial integration one obtains
\begin{align}
\int_{\{0<|m(x)| \}}  g_{|m(x)| }\left(v\right)\mu(dx)
&=
\int_0^\infty g_t(v)dG^*(t)\nonumber\\
&= 
\int_0^{2v} G^*(t)dt\nonumber\\
&\le 
c^* \int_0^{2v} t^{\gamma}dt\nonumber \\
&=
\frac{c^* 2^{1+\gamma}}{1+\gamma } v^{1+\gamma}.
\label{gv}
\end{align}
This together with (\ref{20*}) yields
\begin{align}
I_{n} 
&\le
\frac{c^* 2^{1+\gamma}}{1+\gamma }
\EXP\left\{\left( h\left(\min_{1\le i \le k}U_i\right)  \right)^{1+\gamma}\right\}
\le
c_1h(1/k)^{1+\gamma},
\label{pt}
\end{align}
where in the last step we refer to Jensen's inequality and the condition that the function $h^{1+\gamma}$ is concave. 

Now we consider the {\it estimation error} $J_n$.
We introduce the quantizer
\begin{align*}
Q_{k}(x)
&=
A_{k,j}
\end{align*}
when $x\in A_{k,j}$.
Note that both decision rules $\overline{g}_{k}$ and $g_{\mathrm{NN}}$ are defined in terms of the data $Q_k(X)$ and $\bX'_k$. Based on these coarser data,
$\overline{g}_{k}$ is the Bayes decision rule,
\begin{align}
\label{oB}
\overline{g}_{k}(x)
&=
\mbox{sign } \overline{m}_{k}(x).
\end{align}
For this Bayes decision rule, we now apply \eqref{g} with $X$ being replaced with $Q_k(X)$.
Then, \eqref{g} and \eqref{oB} imply that
\begin{align*}
J_n
&\le 
\EXP\{L(g_n)\}-\EXP\{L(\overline{g}_{k})\}\\
&=\int  \EXP\left\{ 
\IND_{\{ \mbox{sign } m_n(x)\ne \mbox{sign } \overline{m}_{k}(x)\}} \right\} |\overline{m}_{k}(x)|\mu(dx)\\
&=
\sum_{j=1}^k \int_{A_{k,j}}  \EXP\left\{ 
\IND_{\{ \mbox{sign } m_n(x)\ne \mbox{sign } \overline{m}_{k}(x)\}} \right\} |\overline{m}_{k}(x)|\mu(dx)\\
&=
\sum_{j=1}^k\EXP\left\{ 
\IND_{\{ \mbox{sign } \nu_n(A_{k,j})\ne \mbox{sign } \nu(A_{k,j})\}} |\nu(A_{k,j} )| \right\}\\
&\le 
\sum_{j=1}^k\EXP\left\{ 
\IND_{\{ |\nu_n(A_{k,j}) - \nu(A_{k,j})|\ge |\nu(A_{k,j})|\}} |\nu(A_{k,j} )| \right\}.
\end{align*}
We finish the proof by showing that
\begin{align}
\label{wt}
\sum_{j=1}^k\EXP\left\{ 
\IND_{\{ |\nu_n(A_{k,j}) - \nu(A_{k,j})|\ge |\nu(A_{k,j})|\}}  |\nu(A_{k,j} )|\right\}
&=
O\left( \sqrt{\frac{k}{n}} \right).
\end{align}
First, Chebyshev's inequality yields
\begin{align*}
\PROB\{ | \nu_n(A_{k,j})- \nu(A_{k,j})|\ge |\nu(A_{k,j})|\mid \bX'_k\}
&\le 
\min\left\{ 1,\frac{\mu(A_{k,j})}{n\nu(A_{k,j})^2}\right\}\\
&\le
\frac{1}{\sqrt{n}}\frac{\sqrt{\mu(A_{k,j})}}{|\nu(A_{k,j})|}.
\end{align*}
Second, applying Jensen's inequality implies
\begin{align*}
&\sum_{j=1}^k\EXP\left\{ 
\IND_{\{ |\nu_n(A_{k,j}) - \nu(A_{k,j})|\ge |\nu(A_{k,j})|\}}  |\nu(A_{k,j} )|\right\}\\
&\le 
\frac{1}{\sqrt{n}} \sum_{j=1}^k\EXP\left\{ 
 \frac{\sqrt{\mu(A_{k,j})}}{|\nu(A_{k,j})|}|\nu(A_{k,j} )|
\right\}\\
&\le 
\frac{1}{\sqrt{n}}\sum_{j=1}^k\EXP\left\{ \sqrt{\mu(A_{k,j})}\right\}\\
&\le 
\frac{1}{\sqrt{n}}\sum_{j=1}^k \sqrt{\EXP\left\{\mu(A_{k,j})\right\}}\\
&=
\frac{1}{\sqrt{n}}\sum_{j=1}^k \sqrt{1/k}\\
&=
\sqrt{\frac{k}{n}}
\end{align*}
which proves \eqref{wt}.
\end{proof}

\subsection{Proof of Theorem \ref{thm:multi_label_class,priv}}

\begin{proof}
Let
\begin{align*}
\widetilde{m}_{n}(x) 
\doteq \widetilde{m}_{n}(X'_j)
\doteq \frac{\widetilde{\nu}_{n}(A_{k,j})}{\mu(A_{k,j})} \quad \text{if} \;x \in A_{k,j},
\end{align*}
for which we have $\widetilde g_n(x) = \mbox{sign } \widetilde m_{n}(x)$.
Because of (\ref{g}),
\begin{align*}
L(\widetilde g_n) - L^* \leq  \int |\widetilde m_{n}(x) - m(x)| \marg(dx).
\end{align*}
Since
\begin{align*}
\widetilde{m}_{n}(x) \doteq m_{n} (x) + \frac{\sigma_Z \sumi \varepsilon_{i,j}}{n \mu(A_{k,j})} \;\text{ for } x \in A_{k,j},
\end{align*}
by the triangle inequality  we have
\begin{align*}
&\int |\widetilde{m}_{n}(x) - m(x) |\marg(dx)\\
&\leq \int |m_{n}(x) - m(x) |\marg(dx) 
+ \int  |\widetilde{m}_{n}(x) - m_{n}(x) |\marg(dx)\\
&= \int |m_{n}(x) - m(x) |\marg(dx) + 
\sigma_Z\sum_{j=1}^k\bigg| \frac 1n\sum_{i=1}^n\varepsilon_{i,j}\bigg|.
\end{align*}

As in the proof of Theorem 3 in  Gy\"orfi and Weiss \cite{GyWe21}, the first term tends to zero a.s., if $k_n \to \infty$ and $k_n/n\to 0$.
For the second term, under the condition $k_n/\sqrt{n}\to 0$ we prove
\begin{align}
\label{Le} 
\widetilde J_n=\sum_{j=1}^{k_n}  \bigg|\frac{1}{n}\sum_{i=1}^n\varepsilon_{i,j}\bigg| \to 0
\end{align}
a.s.
Similarly to Lemma 1 in  Berrett,  Gy\"orfi and  Walk \cite{BeGyWa21}, for $0<\epsilon\le 1$ and $t=\epsilon n/k_n$, we follow the Chernoff bounding scheme:
\begin{align*}
\PROB\{\widetilde J_n>\epsilon\}
&\le
\frac{\EXP\{ e^{t\widetilde J_n}\}}{e^{t\epsilon}}
=
\frac{\EXP\left\{ e^{ \frac{t}{n}\left|\sum_{i=1}^n\varepsilon_{i,1}\right|  }\right\}}{e^{t\epsilon}}^{k_n}.
\end{align*}
Because of
\begin{align*}
&\EXP\left\{ e^{ \frac{t}{n}\left|\sum_{i=1}^n\varepsilon_{i,1}\right| }\right\}\\
&=
\EXP\left\{ e^{\frac{t}{n}  \sum_{i=1}^n\varepsilon_{i,1} }
\IND_{\{ \sum_{i=1}^n\varepsilon_{i,1}\ge 0\}}\right\}
+
\EXP\left\{ e^{-\frac{t}{n} \sum_{i=1}^n\varepsilon_{i,1} }
\IND_{\{ \sum_{i=1}^n\varepsilon_{i,1}< 0\}}\right\}\\
&\le 
2\EXP\left\{ e^{\frac{t}{n}  \sum_{i=1}^n\varepsilon_{i,1} }\right\}\\
&=
2\EXP\left\{ e^{\frac{t}{n}  \varepsilon_{1,1} }\right\}^n\\
&=
\frac{2}{\left(1-\frac{t^2}{2n^2} \right)^n},
\end{align*}
where the last step holds if $t/n<1\sqrt 2$.
Using the fact that $\log(1-x)\ge -1.4x$ for $0\le x\le 1/8$,
\begin{align*}
\PROB\{\widetilde J_n>\epsilon\}
&\le
\frac{2^{k_n}}{e^{t\epsilon}\left(1-\frac{t^2}{2n^2} \right)^{nk_n}}\\
&=
\frac{2^{k_n}}{e^{\epsilon^2n/k_n}\left(1-\frac{\epsilon^2}{2k_n^2} \right)^{nk_n}}\\
&\le
\frac{2^{k_n}}{e^{\epsilon^2n/k_n-\frac{1.4\epsilon^2n}{2k_n}}}\\
&=
\frac{2^{k_n}}{e^{0.3\epsilon^2n/{k_n}}},
\end{align*}
when $k_n \ge 2$.
Putting $\delta_n=k_n/\sqrt{n}$, $\delta_n\to 0$ holds by assumption and so
\begin{align*}
(\epsilon^2 n/k_n-k_n)/\log n = \sqrt{n}(\epsilon^2/\delta_n-\delta_n)/\log n\to \infty.
\end{align*}
Therefore
\begin{align*}
\sum_n \PROB\{\widetilde J_n>\epsilon\}
&\le 
\sum_n e^{-(0.3\epsilon^2n/k_n-k_n\log 2)}
<
\infty.
\end{align*}
Thus, \eqref{Le} follows from the Borel-Cantelli lemma.
\end{proof}

\subsection{Proof of Theorem \ref{p_patt_rate}}

\begin{proof}
Similarly to the proof of Theorem \ref{patt_rate}, we decompose the excess error probability as
\begin{align*}
\EXP\{L(\widetilde g_n)\}-L^*
&=
I_n+J_n,
\end{align*}
with the approximation error
\begin{align*}
I_n
&=
\EXP\{L(g_{\mathrm{NN}})\}-L^*
\end{align*}
and with estimation error
\begin{align*}
J_n
&=
\EXP\{L(\widetilde g_n)\}-\EXP\{L(g_{\mathrm{NN}})\}.
\end{align*}

The approximation error is the same as in the proof of Theorem \ref{patt_rate}, therefore from (\ref{pt}) one gets
\begin{align*}
I_{n} 
&\le
c_1h(1/k)^{1+\gamma}.
\end{align*}

As to the estimation error, the proof of Theorem \ref{patt_rate} yields that
\begin{align*}
J_n
&\le 
\EXP\{L(\widetilde g_n)\}-\EXP\{L(\overline{g}_{k})\}\\
&\le 
\sum_{j=1}^k\EXP\left\{ 
\IND_{\{ |\widetilde \nu_n(A_{k,j}) - \nu(A_{k,j})|\ge |\nu(A_{k,j})|\}} |\nu(A_{k,j} )| \right\}.
\end{align*}
Chebyshev's inequality yields
\begin{align*}
\PROB\{ |\widetilde \nu_n(A_{k,j})- \nu(A_{k,j})|\ge |\nu(A_{k,j})|\mid \bX'_k\}
&\le 
\min\left\{ 1,\frac{\Var (\widetilde \nu_n(A_{k,j})\mid \bX'_k )}{\nu(A_{k,j})^2}\right\}\\
&=
\min\left\{ 1,\frac{\Var (Z_{1,j}\mid \bX'_k )/n}{\nu(A_{k,j})^2}\right\}\\
&\le
\frac{1}{\sqrt{n}}\frac{\sqrt{\Var (Z_{1,j}\mid \bX'_k ) }}{|\nu(A_{k,j})|}\\
&\le
\frac{1}{\sqrt{n}}\frac{\sqrt{\sigma_Z^2 + \mu(A_{k,j}) }}{|\nu(A_{k,j})|}.
\end{align*}
Therefore, Jensen's inequality yields that
\begin{align*}
&
\sum_{j=1}^k\EXP\left\{ 
\IND_{\{ |\widetilde \nu_n(A_{k,j}) - \nu(A_{k,j})|\ge |\nu(A_{k,j})|\}} |\nu(A_{k,j} )| \right\}\\
&\le 
\frac{1}{\sqrt{n}}\sum_{j=1}^k\EXP\left\{ \frac{\sqrt{\sigma_Z^2 + \mu(A_{k,j}) }}{|\nu(A_{k,j})|}
 |\nu(A_{k,j} )| \right\}\\
&\le 
\frac{1}{\sqrt{n}}\sum_{j=1}^k \sqrt{\sigma_Z^2 + \EXP\left\{\mu(A_{k,j})\right\} }\\
&=
\frac{1}{\sqrt{n}}\sum_{j=1}^k \sqrt{\sigma_Z^2 + 1/k }\\
&=
\frac{k \sqrt{\sigma_Z^2 + 1/k } }{\sqrt{n}}\\
&\approx 
\frac{k \sigma_Z }{\sqrt{n}}.
\end{align*}

\end{proof}

\end{document}